\documentclass[psamsfonts]{conm-p-l}
\usepackage{setspace}
\usepackage{amsfonts}
\usepackage{amsmath}
\usepackage{amssymb}
\usepackage{graphicx}
\allowdisplaybreaks[1]
\newcommand{\ds}{\displaystyle}
\newcommand{\ben}{\begin{enumerate}}
\newcommand{\een}{\end{enumerate}}
\newcommand{\tl}{F^{02d}_1}
\newcommand{\eq}[2][label]{\begin{equation}\label{#1}#2\end{equation}}
\newcommand{\p}{\partial}
\newcommand{\pp}{\p\!}
\newcommand{\BMO}{{\rm{BMO}}}

\copyrightinfo{2000}{American Mathematical Society}

\newtheorem{theorem}{Theorem}[section]
\newtheorem{lemma}[theorem]{Lemma}

\theoremstyle{definition}

\theoremstyle{remark}
\newtheorem{remark}[theorem]{Remark}

\numberwithin{equation}{section}

\begin{document}

\title{Bellman Function and the $H^1-\BMO$ Duality}

\author{Leonid Slavin}
\address{Department of Mathematics, University of Connecticut, Storrs, Connecticut 06269}
%\curraddr{Department of Mathematics and Statistics,
%Case Western Reserve University, Cleveland, Ohio 43403}
\email{slavin@math.uconn.edu}

\author{Alexander Volberg}
\address{Department of Mathematics, Michigan State University,
East Lansing, Michigan 48824}
\email{volberg@math.msu.edu}
\thanks{Volberg's research supported in part by the National Science Foundation grant DMS-0501067.}

%    General info
\subjclass[2000]{Primary 42B30, 42B35}
%\date{January 22, 2006}


\keywords{Bellman function method, $H^1-\BMO$ duality}

\begin{abstract}
A Bellman function approach to Fefferman's $H^1-\BMO$ duality theorem is presented. One Bellman-type argument is used to handle two different one-dimensional cases, dyadic and continuous. An explicit estimate for the constant of embedding $\BMO\subset (H^1)^*$ is given in the dyadic case. The same Bellman function is then used to establish a multi-dimensional analog.
\end{abstract}

\maketitle

\section*{Introduction} 
The emergence in the past decade of the Bellman function method as a powerful and versatile harmonic analysis technique has been characterized by rapid theoretical development on the one hand and somewhat {\it ad hoc}, if effective, approaches to some problems on the other. From the groundbreaking applications in \cite{haar,nt,haar1}, which put the method on the map, to the concerted effort at tracing its origin to stochastic control and building a library of results in \cite{vol, vol1} (see also multiple references therein; in addition, in \cite{vol} an earlier result of Burkholder \cite{burkholder} was put in a Bellman-function framework), to recent explicit computation of actual Bellman functions (and not just their majorants) in \cite{vasyunin, melas, vv, sv} -- the technique has been established as one with many appearances and broad applicability.

In this paper, we seek to reinforce this notion by using the Bellman function method in an unusual setting. Namely, we prove one, the more technically involved, direction of the famous Fefferman $H^1-\BMO$ duality theorem (\cite{fefferman}). The proofs we present are Bellman-function-type proofs (see the discussion in \cite{s}), whereas no extremal problem is posed and thus no Bellman function as such exists. Nonetheless, the main feature of any Bellman-function proof, an induction-by-scales argument, is central to our reasoning. (In Bellman-type arguments, the function on which the induction by scales is performed is commonly referred to as the Bellman function.) It is also worth noting that Bellman proofs often yield explicit (even sharp) constants in inequalities, one reason many well-known results have been reexamined recently with the use of the technique.    

We first consider two cases, dyadic and continuous, in the one-dimensional setting. In the dyadic case, we show that $\BMO^d\subset\left(\tl\right)^*$ (with an explicit estimate for the constant of embedding), with the Triebel-Lizorkin space $\tl$ giving a convenient characterization for $H^1_d(\mathbb{T}),$ the dyadic version of $H^1(\mathbb{T}).$ A simple argument demonstrates the converse inclusion. In the continuous case, we establish the fact that $\ds BMO_0(\mathbb{T})\subset H^1(\mathbb{T})^{*}$ ($\ds BMO_0(\mathbb{T})=\{\varphi\in \BMO(\mathbb{T}),\varphi(0)=0\},$ and as usual, $\varphi(z)$ is the harmonic continuation of $\varphi$ into $\mathbb{D}$). The key to the proofs is a lemma whose hypotheses include the existence of a certain function, one we will call {\it the} Bellman function, slightly abusing the language, since we make no claim as to its uniqueness. We then generalize the continuous-case proof to show that $\BMO(\mathbb{R}^n)\subset \left(\mathcal{H}^1(\mathbb{R}^n)\right)^*.$ This, notably, requires no new tools (except for the natural reformulation of the key lemma in higher-dimensional terms) -- we even use the same Bellman function. Furthermore, we again get an explicit estimate for the constant of embedding. We start by stating our key lemma in the case of an interval-based dyadic lattice. Although, formally, it is a special case of the higher-dimensional lemma, the latter is just its minor modification. 
\section{The formulation of Key Lemma 1.} 
Let $D=D_{I_0}$ be the dyadic lattice rooted in an interval $I_0.$ For an interval $I \in D,$ let $I_-$ and $I_+$ be its left and right halves, respectively.  Consider two functions, $S: D\to [0,\infty)$ and $M: D\to [0,\bar{M}],$ such that 
\eq[d01]{
S_{I_-}=S_{I_+}\ge S_I~~~~\text{and}~~~~ 
M_I\ge \frac12(M_{I_-}+M_{I_+}),\forall I\in D. 
}
\begin{lemma}[Key Lemma 1]Let $S$ and $M$ be as above. Assume there exists a $C^2\!\!$-function $B: [0,\infty)\times [0,\bar{M}]\to \mathbb{R}$ (except, possibly, that $B_x$ or $B_{xx}$ may fail to exist when $x=0$), satisfying
\eq[d02]{
0\le B(x,y)\le 2\bar{M}\sqrt{x},~~~
-\frac{\pp B}{\p x}\frac{\pp B}{\p y}\ge \frac{\bar{M}}2,~~~
\frac{\p^2\! B}{\p x^2}\le 0,~~~\frac{\p^2\! B}{\p y^2}\ge 0,~~~B(0,y)=0.
}
Then, for any positive integer $n,$
\eq[d03]{
\sum_
{\substack{ 
J\in D\\
|J|\ge2^{{-n+1}} 
}}
|J|\sqrt{(S_{J_+}-S_J)\left(M_J-\frac12(M_{J_-}+M_{J_+})\right)}
\le
\sqrt{2\bar{M}}~2^{-n}
\sum_
{\substack{ 
J\in D\\
|J|=2^{{-n}} 
}}
\sqrt{S_J}.
}
\end{lemma}
We will prove the lemma and demonstrate our Bellman function later. For now, we will establish the main results.
\section{The dyadic case} 
Consider the dyadic lattice $D=D_\mathbb{T}$ on $\mathbb{T}.$ For an arc $I\in D,$ let $I_-$ and $I_+$ be its left and right halves, respectively. Also, for a function $f\in L^1(\mathbb{T})$ and $I\in D,$ let $\langle f\rangle_I=\frac1{|I|}\int_If(\theta)\,d\theta.$  Let $\ds\tl$ be the dyadic Triebel-Lizorkin space
\eq[d08]{
\tl=\left\{f\in L^1:~\int_\mathbb{T}\left(\sum_{I\ni \theta;I\in D}\left(\left<f\right>_{I_+}-\left<f\right>_{I_-}\right)^2\right)^{1/2}d\theta<\infty \right\}
} 
with the norm
$$
\|f\|_{\tl}=\int_\mathbb{T}\left(\sum_{I\ni \theta;I\in D}\left(\left<f\right>_{I_+}-\left<f\right>_{I_-}\right)^2\right)^{1/2}d\theta.
$$
%Here, as usual, $a_I$ is a dyadic $H^1$ atom, i.e. a function supported on the arc $I\in D,$ such that $\ds |a|\le %\frac1{|I|},~a.e.,$ and $\ds \int_I\,a(\theta)\,d\theta=0.$ 
We introduce the $L^2\!\!$-based $\BMO^d(\mathbb{T})$ (different from, but equivalent to, the original definition in \cite{jn})
\eq[t1]{
\BMO^d=\left\{\varphi\in L^2: \int_J|\varphi(t)-\left<\varphi\right>_J|^2dt\le C^2|J|, \forall~J\in D\right\}
}
with the best such $C$ being the corresponding norm of $\varphi.$ This definition can be rewritten in a more useful form
\eq[t2]{
\BMO^d=\left\{\varphi\in L^2: \langle\varphi^2\rangle_J-\langle\varphi\rangle^2_J\le C^2, \forall J\in D\right\}.
}
%with the norm
%$$
%\|\varphi\|_{\BMO^d}=\left(\sup_{J\in D}\left\{\left<\varphi^2\right>_J-\left<\varphi\right>^2_J\right\}\right)^{1/2}.
%$$
Definition (\ref{t2}) proved extremely useful in \cite{sv}, but for the purposes of this paper we refashion it in terms of the Haar coefficients of $\varphi.$ Namely, we have
\eq[d06]{
\BMO^d=\left\{\varphi\in L^1:
~\sup_{J\in D}\frac1{|J|}\sum_{I\in D;I\subset J}\left(\left<\varphi\right>_{I_+}-\left<\varphi\right>_{I_-}\right)^2|I|<\infty\right\}
} 
with the norm
$$
\| \varphi\|_{\BMO^d}=\sup_{J\in D}\left(\frac1{|J|}\sum_{I\in D;I\subset J}\left(\left<\varphi\right>_{I_+}-\left<\varphi\right>_{I_-}\right)^2|I|\right)^{1/2}.
$$
To see the equivalence of the definitions (\ref{t1}) and (\ref{d06}), recall the Haar system: for every dyadic arc $I,$ let
$$
h_I=\left\{
\begin{array}{ll}
\phantom{-}\dfrac1{\sqrt{|I|}}&\text{ on }I_-\\
-\dfrac1{\sqrt{|I|}}&\text{ on }I_+\\
~~~0&\text{ elsewhere}
\end{array}.\right.
$$
It is easy to check that $\{h_I\}_{I\in D;I\subset J}$ form an orthonormal system in $L^2_0(J)=\{f\in L^2(J):\int_Jf(\theta)\,d\theta=0\},$ for any $J\in D;$ what is more, the Haar system actually is a basis for $L^2_0(J).$ For any function $f\in L^1$ and every $I\in D$ one can compute the corresponding Haar
coefficient, 
$$
(f,h_I)=\frac{\sqrt{|I|}}2\left(\left<f\right>_{I_-}-\left<f\right>_{I_+}\right).
$$
For $f\in L^2(J)$ we then have $f-\langle f\rangle_J=\sum_{I\in D;I\subset J}(f,h_I)h_I$ and 
$$
\|f-\langle f\rangle_J\|^2_{L^2}=\sum_{I\in D;I\subset J}(f,h_I)^2=\sum_{I\in D;I\subset J}\frac{|I|}4\left(\left<f\right>_{I_-}-\left<f\right>_{I_+}\right)^2.
$$
We state our main result.
\begin{theorem}
$\ds \BMO^d=\left(\tl\right)^*.$  
\end{theorem}
{\sc Proof.} The more difficult inclusion is handled using the Bellman-function lemma stated above.
\begin{lemma}$\ds \BMO^d\subset \left(\tl\right)^*.$ More precisely, in terms of the Haar coefficients, for every $\ds \varphi\in \BMO^d$ and $f\in \tl$,
\begin{align}
\label{d09}
\sum_{J\in D}|(f,h_J)|\,|(\varphi,h_J)|&=\frac14\sum_{J\in D}|J|\,|\left<f\right>_{J_+}-\left<f\right>_{J_-}|\,|\left<\varphi\right>_{J_+}-\left<\varphi\right>_{J_-}|\\*
\notag&\le
\frac{\sqrt2}4\|\varphi\|_{\scriptscriptstyle \BMO^d}\,\| f \|_{\tl}.
\end{align}
\end{lemma}
\begin{proof}Fix $\varphi\in \BMO^d,~f\in \tl.$ For every $J\in D$ define
$$
M_J=\frac1{|J|}\sum_{I\subset J}\left(\left<\varphi\right>_{I_+}-\left<\varphi\right>_{I_-}\right)^2|I|.
$$ 
Then $0\le M_J\le \bar{M}\stackrel{def}{=}\| \varphi \|^2_{\BMO^d}$ and
$M_J-\frac12\left(M_{J_+}+M_{J_-}\right)=\left(\left<\varphi\right>_{J_+}-\left<\varphi\right>_{J_-}\right)^2.$
Define 
$$
S_J=\sum_{I\supsetneq J}\left(\left<f\right>_{I_+}-\left<f\right>_{I_-}\right)^2.
$$ 
Then $S_{J_+}=S_{J_-}=\sum_{I\supset J}\left(\left<f\right>_{I_+}-\left<f\right>_{I_-}\right)^2$ and 
$S_{J_+}-S_J=\left(\left<f\right>_{J_+}-\left<f\right>_{J_-}\right)^2.$ We thus see that the conditions (\ref{d01}) of the lemma are satisfied.
\par

Assuming the existence of the function $B$ in the lemma and using (\ref{d03}), we obtain
\begin{align*}
\sum_
{\substack{ 
J\in D\\
|J|\ge2^{{-n+1}} 
}}
\!\!\!\!\!|J|~&|\left<f\right>_{J_+}-\left<f\right>_{J_-}||\left<\varphi\right>_{J_+}-\left<\varphi\right>_{J_-}|\\
&\le
\sqrt{2\bar{M}}~2^{-n}
\!\!\!\!\!\sum_
{\substack{ 
J\in D\\
|J|=2^{{-n}} 
}}
\sqrt{\sum_{I\supsetneq J}\left(\left<f\right>_{I_+}-\left<f\right>_{I_-}\right)^2}
=\sqrt{2\bar{M}}\int_{\mathbb{T}}\varphi_n(\theta)\,d\theta,
\end{align*}
where $\varphi_n$ is the step function, $\varphi_n(\theta)=
\sqrt{\sum_{I\supsetneq J}\left(\left<f\right>_{I_+}-\left<f\right>_{I_-}\right)^2},$ $\theta\in J$ for each $J\in D$ of length $2^{-n}.$ Since $\varphi_n(\theta)\to
\sqrt{\sum_{I\ni \theta;I\in D}\left(\left<f\right>_{I_+}-\left<f\right>_{I_-}\right)^2},a.e.$ and $f\in\tl,$ letting $n\to \infty$ yields the statement (\ref{d09}) of the theorem, by dominated convergence.
\end{proof}
The proof of the converse inclusion is standard; we include it here for the sake of completeness. 
\begin{lemma}
$\ds \left(\tl\right)^*\subset \BMO^d.$
\end{lemma}
\begin{proof} 
We want to show that for every continuous linear functional $l$ on $\tl$ there exists $\varphi\in \BMO^d$ such that
\eq[d063]{
\|\varphi\|_{\BMO^d}\le c\|l\|
}
and
\eq[d064]{
l(f)=\int_{\mathbb{T}}\varphi(\theta)f(\theta)\,d\theta,~~~\forall f\in\tl.
}
First, we observe that $L^2_0\subset\tl.$ Indeed, for $f\in L^2_0,$
\begin{align*}
\|f\|^2_{\tl}&=\left(\int_\mathbb{T}\left(\sum_{I\ni \theta;I\in D}\left(\left<f\right>_{I_+}-\left<f\right>_{I_-}\right)^2\right)^{\frac12}\!\!\!d\theta\right)^2\le
\int_\mathbb{T}\!\sum_{I\ni \theta;I\in D}\!\!\left(\left<f\right>_{I_+}-\left<f\right>_{I_-}\right)^2\!\!d\theta\\
&=\sum_{I\in D}\int_{\mathbb{T}}\chi_I(\theta)\left(\left<f\right>_{I_+}-\left<f\right>_{I_-}\right)^2d\theta
=\sum_{I\in D}|I|\left(\left<f\right>_{I_+}-\left<f\right>_{I_-}\right)^2=4\|f\|^2_{L^2_0}.
\end{align*}
Let $l\in\left(\tl\right)^*.$ We can apply the Riesz representation theorem to $\left.l\right|_{L^2_0}$ and conclude that there exists a function $\varphi\in L^2_0$ such that 
\eq[d067]{
l(f)=\int_{\mathbb{T}}\varphi(\theta)f(\theta)\,d\theta,~~~\forall f\in L^2_0.
}
We test $l$ on appropriate elements of $L^2_0$ to see that $\varphi\in \BMO^d.$ Let $a_I$ be an atom associated with a dyadic arc $I,$ i.e. be supported on $I$ with $|a_I|\le\frac1{|I|},~a.e.$ and $\int_I a(\theta)\,d\theta=0.$ We have
\begin{align*}
\|a_I\|_{\tl}&=\int_{\mathbb{T}}\left(\sum_{J\ni \theta;J\in D}
\left(\left<a_I\right>_{J_+}-\left<a_I\right>_{J_-}\right)^2\right)^{1/2}d\theta\\
&=\int_I\left(\sum_{J\ni \theta;J\subset I}
\left(\left<a_I\right>_{J_+}-\left<a_I\right>_{J_-}\right)^2\right)^{1/2}d\theta\\
&\le\left(\int_I\sum_{J\ni \theta;J\subset I}
\left(\left<a_I\right>_{J_+}-\left<a_I\right>_{J_-}\right)^2d\theta\right)^{1/2}
\left(\int_I 1\,d\theta\right)^{1/2}\\
&=2\|a_I\|_{L^2_0}\sqrt{|I|}\le\frac2{\sqrt{|I|}}\sqrt{|I|}=2,
\end{align*}
and hence
$$
\left|\int_I\left(\varphi-\left<\varphi\right>_I\right)a_I\right|=\left|\int_{\mathbb{T}}\varphi\,a_I\right|=
|l(a_I)|\le\|l\|\|a_I\|_{\tl}\le2\|l\|.
$$
Since this is true for any atom $a_I,$ we conclude that $\int_I|\varphi-\left<\varphi\right>_I|\le2\|l\||I|$ and thus 
that $\varphi\in \BMO^d$ with the norm estimate (\ref{d063}). Here we have used the equivalence of the $L^1\!$- and $L^2\!$-based $\BMO$ norms, which is due to the John-Nirenberg inequality. The proof of Lemma 2.3 (and hence Theorem 2.1) thus depends on proving that $L^2_0$ is dense in $\tl.$ Together with (\ref{d067}) this will yield the result. 
\par

Take $f\in \tl.$ Let $f_n$ be the truncation of its Haar expansion at the $n\!$-th generation of the dyadic lattice, 
$$
f_n=\sum_{\substack{J\in D\\|J|\ge2^{{-n}}}}(f,h_J)h_J.
$$ 
While $\{f_n\}$ may not converge in the $L^2_0\!$-norm, we show that it does converge\linebreak (to $f$) in the $\tl\!$-norm. We have
$$
\|f-f_n\|_{\tl}=\int_{\mathbb{T}}\left(\sum_{J\ni\theta}\frac4{|J|}(f-f_n,h_J)^2\right)^{1/2}d\theta=
\int_{\mathbb{T}}\Biggl(\sum_{\substack{J\ni\theta\\|J|<2^{{-n}}}}\frac4{|J|}(f,h_J)^2\Biggr)^{1/2}d\theta.
$$
Since $f\in\tl,$ the dominated convergence theorem applies, so $\|f-f_n\|_{\tl}\to0$ as $n\to\infty.$ This concludes the proof of Lemma 2.3 and Theorem 2.1.
\end{proof}
\section{The continuous case}
We define $H^1=H^1(\mathbb{T})$ using the area integral (see, for instance, \cite{stein}); specifically
\eq[d2]{
H^1=\left\{f\in L^1: 
\int_\mathbb{T}\left(\int_{\Gamma_\alpha(e^{i\theta})}|f'(\xi)|^2dA(\xi)\right)^{1/2}d\theta<\infty\right\}
}
with the norm 
$$
\| f\|_{H^1}=\int_\mathbb{T}\left(\int_{\Gamma_\alpha(e^{i\theta})}|f'(\xi)|^2dA(\xi)\right)^{1/2}d\theta. 
$$
Here 
$f(z)$ is the harmonic extension of $f$ into $\mathbb{D}.$ $\ds \Gamma_\alpha(e^{i\theta})$ is the cone-like region with vertex $e^{i\theta}:$
$\ds \Gamma_\alpha(e^{i\theta})=\left\{z\in \mathbb{D}:~\frac{|e^{i\theta}-z|}{1-|z|}< \frac1{\sin\alpha}\right\}$ (see Fig.1). For our purposes, the angle $\alpha$ must be small enough; we will make this more precise shortly.

\begin{figure}[ht]
  \centering{\includegraphics{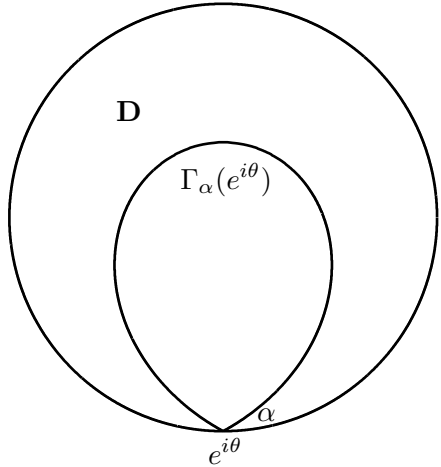}}
%\begin{picture}(1,1)
%\put(-130,135){$\mathbb{D}$}
%\put(-110,92){$\Gamma_\alpha(e^{i\theta})$}
%\put(-80,3){$\alpha$}
%\put(-100,-12){$e^{i\theta}$}
%\end{picture}
%\medskip
%
\caption{The region $\Gamma_\alpha(e^{i\theta}).$}
\end{figure}

%\begin{figure}[ht]
%\begin{picture}(200,200)
%\thicklines
%\put(100,100){\circle{300}}
%\end{picture}
%\medskip
%
%\caption{The region $\Gamma_\alpha(e^{i\theta}).$}
%\end{figure}
%
%
The corresponding definition of $\BMO_0=\BMO_0(\mathbb{T})$ is 
\eq[d3]{
\BMO_0=\left\{\varphi\in L^1:\sup_{{\rm arc~}I\subset \mathbb{T}}\frac1{|I|}\int_{Q_I}|\varphi'(\xi)|^2(1-|\xi|)\,dA(\xi)<\infty,
~\varphi(0)=0\right\},
}
where $\varphi(z)$ is the harmonic extension of $\varphi$ into $\mathbb{D}$ and $Q_I$ is the Carleson square corresponding to the arc $I,$ $Q_I=\{z\in \mathbb{D}:~z/|z|\in I, |z|\ge1-|I|\}.$ The norm in this space is then 
$$
\ds \| \varphi\|_{\BMO_0}=
\sup_{{\rm arc~}I\subset \mathbb{T}}\left(\frac1{|I|}\int_{Q_I}|\varphi'(\xi)|^2(1-|\xi|)dA(\xi)\right)^{1/2}.
$$
We are now in a position to state the main result.
\begin{theorem}
$\ds \BMO_0\subset \left(H^1\right)^*.$ More precisely,
\eq[d5]{
\left|\int_\mathbb{T}\varphi(e^{i\theta})\bar{f}(e^{i\theta})\,d\theta\right|\le C\| \varphi \|_{\BMO_0}\| f\|_{H^1},
~\forall \varphi\in \BMO_0, \forall f\in H^1.
}
\end{theorem}
\begin{proof}
Not surprisingly, the proof starts with a dyadic construction.
For every $J\in D=D_\mathbb{T}$ define 
$$
M_J=\frac1{|J|}\int_{Q_J}|\varphi'(\xi)|^2(1-|\xi|)\,dA(\xi).
$$ 
Clearly,
$\ds M_J\le \bar{M}\stackrel{def}{=}\| \varphi \|^2_{\BMO_0}.$ We have 
$$
M_J-\frac12\left(M_{J_+}-M_{J_-}\right)=\frac1{|J|}\int_{TQ_J}|\varphi'(\xi)|^2(1-|\xi|),dA(\xi).
$$ 
Here $TQ_J$ is the top half of the (dyadic) square $Q_J,$ $TQ_J=Q_J\backslash(Q_{J_+}\cup Q_{J_-})$ (see Fig.2).
\begin{figure}[ht]
\begin{center}
\begin{picture}(100,100)
\thicklines
\put(0,0){\line(1,0){100}}
\put(100,0){\line(0,1){100}}
\put(100,100){\line(-1,0){100}}
\put(0,100){\line(0,-1){100}}
\put(0,50){\line(1,0){100}}
\put(50,0){\line(0,1){50}}
\put(0,25){\line(1,0){100}}
\put(25,0){\line(0,1){25}}
\put(0,12){\line(1,0){50}}
\put(12,0){\line(0,1){12}}
\put(48,-19){$J$}
\put(22,-10){$J_-$}
\put(72,-10){$J_+$}
\put(43,70){$TQ_J$}
\put(18,30){$TQ_{J_-}$}
\put(68,30){$TQ_{J_+}$}
\end{picture}
\vspace{1cm}

\caption{The decomposition $\ds Q_J=\bigcup_{I\in D,I\subset J}TQ_I.$}
\end{center}
\end{figure}
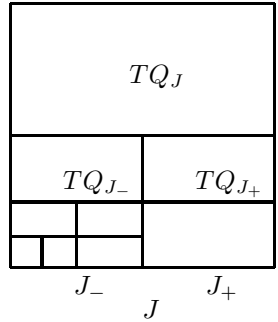

Define 
$$
S_J=\int_{\Gamma^d_J}|f'(\xi)|^2dA(\xi)=\sum_{I\supsetneq J; I\in D}\int_{TQ_I}|f'(\xi)|^2dA(\xi).
$$ 
Here $\Gamma^d_J$ is the dyadic cone, $\Gamma^d_J=\bigcup_{I\supsetneq J}TQ_J.$ We observe that there exists a critical value $\alpha_0>0$ that $\ds\Gamma^d_J\subset \Gamma_\alpha(e^{i\theta}),~\forall \theta\in J, 0<\alpha\le\alpha_0.$ For instance, if $0<\alpha<1/20,$ the inclusion holds. Fix any such $\alpha.$
\medskip

We have $S_{J_-}=S_{J_+}=\sum_{I\supseteq J; I\in D}\int_{TQ_I}|f'(\xi)|^2dA(\xi)$ and thus,
$S_{J_-}-S_J=S_{J_+}-S_J=\int_{TQ_J}|f'(\xi)|^2dA(\xi).$ Therefore, the conditions (\ref{d01}) of the key lemma are satisfied.
Assuming the existence of the function $B$ in the lemma and using (\ref{d03}), we have
\begin{align*}
\sum_{|J|\ge 2^{{-n+1}}}|J|&\left(\frac1{|J|}\int_{TQ_J}|\varphi'(\xi)|^2(1-|\xi|)\,dA(\xi)\right)^{1/2}
\left(\int_{TQ_J}|f'(\xi)|^2dA(\xi)\right)^{1/2}\\
&\le\sqrt{2\bar{M}}~2^{-n}\sum_{|J|=2^{{-n}}}\left(\int_{\Gamma^d_J}|f'(\xi)|^2dA(\xi)\right)^{1/2}.
\end{align*}
Let us estimate the left-hand side as $n\to\infty.$
\begin{align*}
\lim_{n\to\infty}&\sum_{|J|\ge 2^{{-n+1}}}|J|\left(\frac1{|J|}\int_{TQ_J}|\varphi'(\xi)|^2(1-|\xi|)\,dA(\xi)\right)^{1/2}
\left(\int_{TQ_J}|f'(\xi)|^2dA(\xi)\right)^{1/2}\\
&\ds=\sum_{J\in D}|J|^{1/2}\left(\int_{TQ_J}|\varphi'(\xi)|^2(1-|\xi|)\,dA(\xi)\right)^{1/2}
\left(\int_{TQ_J}|f'(\xi)|^2dA(\xi)\right)^{1/2}\\
&\ds\ge\sum_{J\in D}|J|^{1/2}\int_{TQ_J}|\varphi'(\xi)|\,|f'(\xi)|(1-|\xi|)^{1/2}dA(\xi)\\
&\ds\ge C'\sum_{J\in D}|J|\int_{TQ_J}|\varphi'(\xi)|\,|f'(\xi)|\,dA(\xi)\\
&\ds\ge C'\int_\mathbb{D}|\varphi'(\xi)|\,|f'(\xi)|\log{\frac1{|\xi|}}\,dA(\xi)\\*
\intertext{\noindent
(Here we have used the fact that $\ds (1-|\xi|)^{1/2}\sim |J|^{1/2}$ and $\ds |J|\sim \log{\frac1{|\xi|}}$ if $\xi\in TQ_J.$ In addition, $\ds \bigcup_{J\in D}TQ_J=\mathbb{D}.$)}
&\ds\ge C'\left|\int_\mathbb{D}\p \varphi \bar{\p} \bar{f} \log{\frac1{|\xi|}}dA(\xi)\right|\\
&\ds=C''\left|\int_\mathbb{D}\Delta(\varphi \bar{f})\log{\frac1{|\xi|}}dA(\xi)\right|,
\end{align*}
where we have used the fact that $\ds \p \varphi \bar{\p} \bar{f}=\p\bar{\p}(\varphi \bar{f})=\frac14\Delta(\varphi \bar{f}),$ since $\varphi$ and $f$ are analytic.
\par

Recall Green's formula
$$
\frac1{2\pi}\int_\mathbb{T}F(e^{i\theta})d\theta-F(0)=\frac1{2\pi}\int_\mathbb{D}\Delta F(\xi)\log\frac1{|\xi|}dA(\xi).
$$
Since $\varphi(0)=0,$ we get $\ds \lim_{n\to \infty}(LHS)\ge C\left|\int_\mathbb{T}
\varphi(e^{i\theta})\bar{f}(e^{i\theta})d\theta\right|.$
\par

On the right-hand side we obtain, as $n\to\infty,$ 
$$
\sqrt2\,\|\varphi\|_{\BMO_0}
\int_\mathbb{T}\left(\int_{\Gamma^d_\alpha(e^{{i\theta}})}|f'(\xi)|^2dA(\xi)\right)^{1/2}d\theta,
$$
where $\ds \Gamma^d_\alpha(e^{{i\theta}})=\bigcup_{J\ni e^{i\theta}}\Gamma^d_J.$ Since each 
$\ds \Gamma^d_J\subset \Gamma_\alpha(e^{i\theta}),$ we have 
$\ds \Gamma^d_\alpha(e^{i\theta})\subset \Gamma_\alpha(e^{i\theta}),$ and thus
$$
\int_\mathbb{T}\left(\int_{\Gamma^d_\alpha(e^{{i\theta}})}|f'(\xi)|^2dA(\xi)\right)^{1/2}d\theta\le
\int_\mathbb{T}\left(\int_{\Gamma_\alpha(e^{{i\theta}})}|f'(\xi)|^2dA(\xi)\right)^{1/2}d\theta=\| f\|_{H^1}.
$$
Putting together the estimates for the right- and left-hand sides, we obtain the statement (\ref{d5}).
\end{proof}
\section{Multi-dimensional setting} 
In this section, we first reformulate conditions (\ref{d01}) and conclusion (\ref{d03}) of Lemma 1.1 in terms of higher-dimensional dyadic lattices. (Observe that conditions (\ref{d02}) on the function $B$ do not change, and so the same function can be used in any dimension.) We then prove a multi-dimensional analog of Theorem 3.1.

Let $D=D_{P}$ be the dyadic lattice rooted in a cube $P\subset \mathbb{R}^n.$ For a cube $I\in D,$ let $I^1,I^2,...,I^{2^n}$ be its dyadic offspring, that is the $2^n$ disjoint dyadic subcubes of $I$ of size $2^{-n}|I|$.  Consider two functions, $S: D\to [0,\infty)$ and $M: D\to [0,\bar{M}],$ such that 
\eq[d001]{
S_{I^1}=S_{I^2}=...=S_{I^{2^n}}\ge S_I~~~~\text{and}~~~~ 
M_I\ge 2^{-n}\sum_{v=1}^{2^n}M_{I^v},\forall I\in D. 
}
\begin{lemma}[Key Lemma 2]Let $S$ and $M$ be as above. Assume there exists a $C^2\!\!$-function $B: [0,\infty)\times [0,\bar{M}]\to \mathbb{R}$ (except, possibly, that $B_x$ or $B_{xx}$ may fail to exist when $x=0$), satisfying
\eq[d002]{
0\le B(x,y)\le 2\bar{M}\sqrt{x},~~~
-\frac{\pp B}{\p x}\frac{\pp B}{\p y}\ge \frac{\bar{M}}2,~~~
\frac{\p^2\! B}{\p x^2}\le 0,~~~\frac{\p^2\! B}{\p y^2}\ge 0,~~~B(0,y)=0.
}
Then, for any positive integer $m,$
\eq[d003]{
\sum_
{\substack{ 
J\in D\\
|J|\ge2^{-n(m-1)} 
}}
|J|\sqrt{(S_{J^1}-S_J)\left(M_J-2^{-n}\sum_{v=1}^{2^n}M_{I^v}\right)}
\le
\sqrt{2\bar{M}}~2^{-nm}\!\!\!\!
\sum_
{\substack{ 
J\in D\\
|J|=2^{{-nm}} 
}}
\sqrt{S_J}.
}
\end{lemma}

We prove Lemma 4.1 in the next section. Now, we introduce the appropriate analogs of (\ref{d2}) and (\ref{d3}). Fix a dyadic lattice $D$ on $\mathbb{R}^n.$ For $I\in D$ with side length $l(I),$ define $Q_I=I\times (0,l(I)].$ As before, let $TQ_I$ be the top half of $Q_I,$ i.e. $TQ_I=Q_I/\bigcup_{v=1}^{2^n}Q_{I^v}.$  Given $x\in \mathbb{R}^n,$ introduce the ``strange'' cones
$$
\Gamma_x^d=\bigcup_{I\ni x,I\in D}TQ_I.
$$
For $f\in L^1(\mathbb{R}^n),$ let $f(y,t)$ be its harmonic extension into $\mathbb{R}_+^{n+1}.$ We use an area-integral-like characterization of $H^1(\mathbb{R}^n)$
$$
\mathcal{H}^1(\mathbb{R}^n)=\left\{f\in L^1(\mathbb{R}^n):~\int_{\mathbb{R}^n}\left(\int_{\Gamma_x^d}
|\nabla f(y,t)|^2t^{1-n}\,dy\,dt\right)^{1/2}dx<\infty\right\}
$$
with the corresponding ``strange'' norm
\eq[d111]{
\|f\|_{\mathcal{H}^1}=\int_{\mathbb{R}^n}\left(\int_{\Gamma_x^d}
|\nabla f(y,t)|^2t^{1-n}\,dy\,dt\right)^{1/2}dx.
}
We also use a natural $\BMO(\mathbb{R}^n).$
$$
\BMO(\mathbb{R}^n)=\left\{\varphi\in L^1(\mathbb{R}^n):~\int_I |\nabla \varphi(y,t)|^2t\,dy\,dt\le C^2|I|,\forall I\in D\right\}.
$$
The best such $C$ is the corresponding $\BMO$ norm. In this notation we can state the following theorem 
\begin{theorem}
$\BMO(\mathbb{R}^n)\subset \left(\mathcal{H}^1(\mathbb{R}^n)\right)^*$ and the constant of embedding does not depend on dimension.
\end{theorem}
\begin{proof}
The proof closely parallels that of Theorem 3.1. For $J\in D,$ let $\Gamma^d_J=\bigcup_{I\supsetneq J}TQ_J.$ Setting
$$
S_J=\int_{\Gamma^d_J}|\nabla f(y,t)|^2t^{1-n}dy\,dt,~~~M_J=\frac1{|J|}\int_J|\nabla\varphi(y,t)|^2t\,dy\,dt
$$
and using Lemma 4.1 (for $f,\varphi$ with finite support, since the lemma works with a cube-based lattice), we get
\begin{align*}
\sum_{|J|\ge 2^{-n(m-1)}}|J|&\left(\frac1{|J|}\int_{TQ_J}|\nabla\varphi(y,t)|^2t\,dy\,dt\right)^{1/2}
\left(\int_{TQ_J}|\nabla f(y,t)|^2t^{1-n}dy\,dt\right)^{1/2}\\
&\le\sqrt{2\bar{M}}~2^{-nm}\sum_{|J|=2^{{-nm}}}\left(\int_{\Gamma^d_J}|\nabla f(y,t)|^2t^{1-n}dy\,dt\right)^{1/2}.
\end{align*}
The right-hand side goes to $\sqrt2\,\|\varphi\|_{\BMO}\|f\|_{\mathcal{H}^1}$ as $m\to\infty.$ On the left we have
\begin{align*}
\sum_{J\in D}|J|&\left(\frac1{|J|}\int_{TQ_J}|\nabla\varphi(y,t)|^2t\,dy\,dt\right)^{1/2}
\left(\int_{TQ_J}|\nabla f(y,t)|^2t^{1-n}dy\,dt\right)^{1/2}\\
&\ge\sum_{J\in D}|J|^{1/2}\int_{TQ_J}|\nabla\varphi(y,t)||\nabla f(y,t)|t^{1-n/2}\,dy\,dt\\
&\ge\int_{\mathbb{R}^{n+1}_+}|\nabla\varphi(y,t)||\nabla f(y,t)|t\,dy\,dt\\
&\ge|\int_{\mathbb{R}^{n+1}_+}(\nabla\varphi(y,t))^T(\nabla f(y,t))t\,dy\,dt|.
\end{align*}
since $t^{-n/2}\ge |J|^{-1/2}$ when $t\in TQ_J.$ Integration by parts yields
$$
\left|\int_{\mathbb{R}^{n+1}_+}(\nabla\varphi(y,t))^T(\nabla f(y,t))t\,dy\,dt\right|=
\frac12\left|\int_{\mathbb{R}^n}\varphi(y)f(y)\,dy\right|.
$$ 
Putting the left-hand and right-hand estimates together, we get
\eq[d004]{
\left|\int_{\mathbb{R}^n}\varphi(y)f(y)\,dy\right|\le2\sqrt2\,\|\varphi\|_{\BMO}\|f\|_{\mathcal{H}^1}.
}
\end{proof}
\begin{remark}
The fact that the constant in (\ref{d004}) does not depend on dimension is due to the ``dimensional'' choice of the norm (\ref{d111}). Indeed, for each $n$ there exists an aperture $a=a(n)$ such that $\Gamma_x^d\subset\Gamma_x^a\stackrel{def}{=}\{(y,t)\in\mathbb{R}^n\times\mathbb{R}^+:|y-x|<at\}.$ However, $a(n)\to\infty$ as $n\to\infty.$
\end{remark}
\section{The proof of the key lemmas}
\begin{proof}
We first establish the one-dimensional result and then briefly discuss the changes needed in the higher-dimensional situation.

Fix $J\in D.$ Let $S=S_J;~S_0=S_{J_-}=S_{J_+};~M=M_J;~M_-=M_{J_-};~M_+=M_{J_+}.$ Assume for the moment that $S\ne0.$ Then
\begin{align*}
&\frac12B(S_{J_-},M_{J_-})+\frac12B(S_{J_+},M_{J_+})\\
&=\frac12B(S_0,M_-)+\frac12B(S_0,M_+)-B(S_0,M)+B(S_0,M)-B(S,M)+B(S,M)\\
&\ge
B\left(S_0,\frac12(M_-+M_+)\right)-B(S_0,M)+\frac{\pp B}{\p S}(\hat{S},M)(S_0-S)+B(S,M)\\
&=-\frac{\pp B}{\p M}(S_0,\hat{M})\left(M-\frac12(M_-+M_+)\right)+\frac{\pp B}{\p S}(\hat{S},M)(S_0-S)+B(S,M),\\*
\intertext{
for some $\hat{S}\in (S,S_0)$ and $\hat{M}\in \left(\frac12(M_-+M_+),M\right).$ Since 
$B_{SS}\le 0,$ we have $B_S(\hat{S},M)\ge B_S(S_0,M)$ and 
since $B_{MM}\ge 0,$ we have 
$-B_M(S_0,\hat{M})\ge -B_M(S_0,M).$ Also, the first and third conditions (\ref{d02}) imply that 
$B_S(S_0,\hat{M})\ge0$ and thus, by the second condition, 
$-B_M(S_0,\hat{M})\ge0.$ We continue}
&\ge
-\frac{\pp B}{\p M}(S_0,M)\left(M-\frac12(M_-+M_+)\right)+\frac{\pp B}{\p S}(S_0,M)(S_0-S)+B(S,M)\\
&\ge
2\sqrt{-\frac{\pp B}{\p M}(S_0,M)\,\frac{\pp B}{\p S}(S_0,M)}\sqrt{\left(M-\frac12(M_-+M_+)\right)(S_0-S)}+B(S,M).
\end{align*}
Using the second condition, we get
\begin{align}
\label{al}
\notag\frac12B(S_{J_-},M_{J_-})&+\frac12B(S_{J_+},M_{J_+})\\
&\ge\sqrt{2\bar{M}}\sqrt{\left(M-\frac12(M_-+M_+)\right)(S_0-S)}+B(S,M).
\end{align}
If $S=0,S_0\ne0,$ the argument works with minor corrections, because $B_M$ is continuous at $(0,M).$ If $S=S_0=0,$ (\ref{al}) is trivially true. Now,
\begin{align*}
2^{-n}
&\sum_
{\substack{ 
J\in D\\
|J|=2^{{-n}} 
}}
\frac12B(S_J,M_J)\\
&=2^{-n}
\!\!\!\!\!\!\sum_
{\substack{ 
J\in D\\
|J|=2^{{-n+1}} 
}}
\left[\frac12B(S_{J_-},M_{J_-})+\frac12B(S_{J_+},M_{J_+})\right]\\
&\ge 2^{-n}\!\!\!\!\!\!\sum_
{\substack{ 
J\in D\\
|J|=2^{{-n+1}} 
}}
\left[\sqrt{2\bar{M}}\sqrt{(S_{J_+}-S_J)\left(M_J-\frac12(M_{J_-}+M_{J_+})\right)}+B(S_J,M_J)\right]\\
&=2^{-n}F_n+2^{-n+1}
\!\!\!\!\!\!\sum_
{\substack{ 
J\in D\\
|J|=2^{-n+2} 
}}
\left[\frac12B(S_{J_-},M_{J_-})+\frac12B(S_{J_+},M_{J_+})\right]\\
&\ge 2^{-n}F_n+2^{-n+1}F_{n-1}+2^{-n+1}
\!\!\!\!\!\!\sum_
{\substack{ 
J\in D\\
|J|=2^{-n+2} 
}}
B(S_J,M_J)\\
&\ge\dots\ge\sum_{k=1}^n2^{-k}F_k+\frac12 B(S_{I_0},M_{I_0}),
\end{align*}
where we have set
$$
F_k=
\sum_
{\substack{ 
J\in D\\
|J|=2^{-k+1} 
}}
\sqrt{2\bar{M}}\sqrt{(S_{J_+}-S_J)\left(M_J-\frac12(M_{J_-}+M_{J_+})\right)}.
$$
Using the fact that $B(S_J,M_J)\le 2\bar{M}\sqrt{S_J},\forall J\in D$ and $B(S_{I_0},M_{I_0})\ge 0,$ we get
$$
2^{-n}\!\!\!\!
\sum_
{\substack{ 
J\in D\\
|J|=2^{{-n}} 
}}
\bar{M}\sqrt{S_J}\ge\frac12\sum_{k=1}^n 2^{-k+1}\!\!\!\!\!\!
\sum_
{\substack{ 
J\in D\\
|J|=2^{-k+1} 
}}
\sqrt{2\bar{M}}\sqrt{(S_{J_+}-S_J)\left(M_J-\frac12(M_{J_-}+M_{J_+})\right)},
$$
thus proving the lemma.
\end{proof}
In the case of $D=D_P$ for a cube $P\subset\mathbb{R}^n,$ we observe that since $B_{yy}\ge0,$ we have
$$
2^{-n}\sum_{v=1}^{2^n}B(S_0,M_{J^v})\ge B\left(S_0,2^{-n}\sum_{v=1}^{2^n}M_{J^v}\right)
$$
and so
\begin{align*}
2^{-n}&\sum_{v=1}^{2^n}B(S_0,M_{J^v})\\
&\ge\sqrt{2\bar{M}}\sqrt{\left(M-2^{-n}\sum_{v=1}^{2^n}M_{J^v}\right)(S_0-S)}+B(S,M).
\end{align*}
Setting
$$
F_k=
\sum_
{\substack{ 
J\in D\\
|J|=2^{-n(k-1)} 
}}
\sqrt{2\bar{M}}\sqrt{(S_{J^1}-S_J)\left(M_J-2^{-n}\sum_{v=1}^{2^n}M_{J^v}\right)},
$$
we then get, for a positive integer $m,$
$$
2^{-nm}
\sum_
{\substack{ 
J\in D\\
|J|=2^{{-nm}} 
}}
\frac1{2^n}B(S_J,M_J)\ge\sum_{k=1}^m 2^{-nk}F_k+\frac1{2^n}B(S_P,M_P),
$$
which yields (\ref{d003}).
\section{A sample Bellman function}
To finish the proofs, we need a function $B$ such that 
$$
0\le B\le C_1\sqrt x,\quad
\frac{\partial^2B}{\partial x^2}\le0,\quad
\frac{\partial^2B}{\partial y^2}\ge0,\quad
-\frac{\partial B}{\partial y}\,\frac{\partial B}{\partial x}\ge C_2,\quad
B(0,y)=0,
$$
for some positive constants $C_1, C_2.$ The proof of the key lemma using $B$ suggests that we want to choose these constants in order to minimize the ratio $C_1/\sqrt{C_2}.$ To make the estimates ``sharper,'' we require that
$
B_{yy}=0.
$
(Because of the first condition, we cannot require equality in $B_{xx}\le0$.) This means that $B$ is a linear function of $y.$ Furthermore, $B_y$ must be negative. Because of the first condition and the homogeneity in the way $B$ is used in the key lemmas, we seek $B$ in the form
\eq[s1]{
B(x,y)=\sqrt x(A-y).
}
Therefore,
$$
-\frac{\partial B}{\partial x}\,\frac{\partial B}{\partial y}=\frac{A-y}2\ge\frac{A-\bar{M}}2,
$$
since $y\le\bar{M}.$ We have $C_1=A$ and for the ratio to be minimized 
$$
\frac{C_1}{\sqrt{C_2}}=\frac{\sqrt2A}{\sqrt{A-\bar{M}}}.
$$
The minimum of this ratio is attained at $A=2\bar{M},$ thus producing the function
\eq[s2]{
B(x,y)=\sqrt x(2\bar{M}-y),
}
satisfying conditions (\ref{d03}).
\begin{remark}
While the function (\ref{s2}) is the best function of the form (\ref{s1}), it is unlikely that the explicit constant in (\ref{d03}) is sharp. In the proof of Lemma 1.1, the inequalities $B\le 2\bar{M}\sqrt{S}$ and $-B_xB_y\ge\bar{M}/2$ are used simultaneously. But the former becomes an equality when $M=0,$ while the latter -- when $M=\bar{M}.$ 
\end{remark}
\section*{Conclusion}
The proofs we have presented are elementary and short, demonstrating yet again the efficiency of the Bellman-function-type approach. To obtain sharp results even in the one-dimensional case, however, one ideally would want to pose an extremal problem (actually two different extremal problems, a dyadic and a continuous one) and compute the corresponding Bellman functions. The proper choice of variables (the starting point in the Bellman formalism) is far from clear. Is is possible that sharp constants can be obtained by finding a different (more complex) function satisfying conditions (\ref{s1}). Alternatively, the proof of the key lemma can be modified resulting in a different set of conditions needed. In either case, to prove that the resulting constant is sharp, one needs to consider the exact function $B$ used to establish the lemma. Then one needs to come up with a pair of functions ($\varphi\in \BMO,f\in\tl$ in the dyadic case) or a pair of sequences thereof, such that the induction-by-scales chain of inequalities in the proof of the key lemma becomes a chain of equalities (or asymptotic equalities), when used in conjunction with $B.$  
\bibliographystyle{amsalpha}

\end{document}